\documentclass[11pt]{amsart}
\usepackage[colorlinks=blue,linkcolor=blue]{hyperref}
\usepackage{enumerate}
\usepackage{lineno}
\usepackage{graphicx}
\usepackage{geometry}
\usepackage{enumitem}
\makeatletter
\@namedef{subjclassname@}{%
\textup{}}
\makeatother
\newtheorem{thm}{Theorem}[section]
\newtheorem{thmx}{Theorem}[section]
\newtheorem{cor}[thm]{Corollary}
\newtheorem{lem}[thm]{Lemma}

\newtheorem{case}{Case}
\newtheorem{subcase}{Subcase}

\newtheorem*{conj*}{Denjoy's Conjecture}

\theoremstyle{definition}
\newtheorem{con}{Conjecture}

\newtheorem{rem}{Remark}
\newtheorem{Que}{Question}
\newtheorem{exa}{Example}
\numberwithin{equation}{section}

\begin{document}

\title[]{ On meromorphic solutions of Fermat type delay-differential equations with two exponential terms }

\author{Xuxu Xiang, Jianren Long*, Mengting Xia, Zhigao Qin\quad }

\address{Xuxu Xiang \newline School of Mathematical Sciences, Guizhou Normal University, Guiyang, 550025, P.R. China. }
\email{1245410002@qq.com}

\address{Jianren Long \newline School of Mathematical Sciences, Guizhou Normal University, Guiyang, 550025, P.R. China. }
\email{longjianren2004@163.com}

\address{Mengting Xia \newline School of Mathematical Sciences, Guizhou Normal University, Guiyang, 550025, P.R. China. }
\email{2190331341@qq.com}

\address{Zhigao  Qin  \newline School of Mathematical Sciences, Guizhou Normal University, Guiyang, 550025, P.R. China. }
\email{1480266124@qq.com}
%\address{Ling Wang\newline School of Mathematical Sciences, Guizhou Normal University,
%Guiyang, 550025, P.R. China. }
%\email{2036448551@qq.com}

\date{}

%%............................................................................................................................................................
%%...........................................................................
%??(abstract)...................................................................
%%............................................................................................................................................................

\begin{abstract}The  existence of  the meromorphic  solutions to  Fermat type delay-differential equation
	\begin{equation}
		f^n(z)+a(f^{(l)}(z+c))^m=p_1(z)e^{a_1z^k}+p_2(z)e^{a_2z^k}, \nonumber
	\end{equation}
	is derived by using Nevanlinna theory under certain conditions, where $k\ge1$, $m,$  $n$  and $l$ are integers, $p_i$ are nonzero entire functions of order less than $k$, $c$, $a$ and  $a_i$ are  constants, $i=1,2$.
	These results not only improve the previous results from Zhu et al. [J. Contemp. Math. Anal. 59(2024), 209-219], Qi et al. [Mediterr. J. Math. 21(2024), article no. 122], but also completely solve two conjectures posed by Gao et al. [Mediterr. J. Math. 20(2023), article no. 167].
	Some examples are given to illustrate these results.
\end{abstract}

%\thanks{*Corresponding author}

%%.............................................................................................................................................................
%%.........................................................................
%?????(keywords)....................................................................
%%.............................................................................................................................................................

\keywords{Meromorphic  solutions; Existence; Exponential functions; Nevanlinna  theory; Fermat type delay-differential equations
*Corresponding author. \\
}
%\qquad *Correspording author}
\maketitle
%%..............................................................................................................................................................
%%.....................................................................................
%????(Introduction}.......................................................
%%.................
\section{Introduction  and main results}\label{sec1}

~~~~
Let $f$ be a meromorphic function in the complex plane $\mathbb{C}$. Assume that the reader is familiar with the standard notation and basic results of Nevanlinna theory, such as $m(r,f),~N(r,f)$,$~T(r,f)$, see \cite{hayman} for more details. A meromorphic
function $g$ is said to be a small function of $f$ if $T(r,g)=S(r,f)$, where $S(r,f)$  denotes any quantity
that satisfies $S(r,f)= o(T(r, f))$ as $r$ tends to infinity, outside of a possible exceptional set of finite linear measure. $\rho(f)=\underset{r\rightarrow \infty}{\lim\sup}\frac{\log^+T(r,f)}{\log r}$ and   $\rho_2(f)=\underset{r\rightarrow \infty}{\lim\sup}\frac{\log^+\log^+T(r,f)}{\log r}$  are used to denote the order and the hyper-order  of $f$, respectively. Define ${\Theta(\infty,f)}=1-\overline{\lim\limits_{r\to\infty}}{\frac{\overline{N}\left(r,f\right)}{T(r,f)}}$,  and $\mu(f)$  is used to denote the lower order of $f$.

%In 1995, Wiles \cite{Wiles} proved Fermat's Last Theorem, which was proposed by Fermat in 1637 when he was reading the book Arithmetica: the Diophantine equation $x^n+y^n=z^n$has no positive integer solutions $x, y, z$ when $ n\ge3$. It took more than 300 years from the proposal to the complete proof of Fermat's Last Theorem.

 In 1927,  Montel\cite{Montel} studied Fermat type functional equation
\begin{align}
	\label{0.1.1}
	f^n(z)+g^m(z)=1,
\end{align}
where  $f,g$ are meromorphic functions, $m,n$ are positive integers. Montel proved  \eqref{0.1.1} has no nonconstant entire solutions under the conditions $m=n>2$. Later,  the existence of meromorphic solutions to equation \eqref{0.1.1} is discussed by Gross\cite{Gross1,Gross2} and Baker\cite{B}. Some important results related to Fermat type functional equations are summarized by Chen et al. \cite[Proposition 1]{Chen}.

In 1970, Yang\cite{Yang1970} applied Nevanlinna  theory to investigate the
Fermat functional equation
\begin{align}
	\label{0.1.2}
	a(z)f^n(z)+b(z)g^m(z)=1,
\end{align}
where $a$, $b$ are small functions of $f$ and $g$, and $m,n\ge3$  are integers. Indeed, upon careful
inspection of the proof of \cite[Theorem 1]{Yang1970},  Laine et al.\cite{Laine} expressed Yang’s result  as follows.

\vspace{6pt}
\begin{thmx}\cite[Theorem 1.1]{Laine}
	\label{thb}
	Let $m$ and $n$ be positive integers satisfying $\frac{1}{m}+\frac{1}{n}<\frac{2}{3}$. Then, equation \eqref{0.1.2}
	has no nonconstant meromorphic solutions $f$ and $g$.
	Moreover, if $\frac{1}{m}+\frac{1}{n}<1$, there
	exists no nonconstant meromorphic solution of \eqref{0.1.2} such that ${\Theta(\infty,f)}={\Theta(\infty,g)}=1$.
\end{thmx}
\vspace{6pt}

In 2004,  Yang et al.\cite{yang2004} considered the Fermat type differential equation
\begin{align}
	\label{1.5}
	f^2(z)+a(f'(z))^2=b(z),
\end{align}
proved that if $b(z)$ is not a constant, $a$ is a constant, then \eqref{1.5} has  no nonconstant meromorphic solution $f$ satisfy $T(r,b)=S(r,f)$.

With the establishment and development of the difference Nevanlinna theory \cite{chiang,halburd2006,halburd20062}, the study of difference equations has become a hot topic, see\cite{halburd2007,halburd2017,lk,Lw, long, long1,xxx} and references therein.
Recently, there are many explorations concerning the Fermat type
difference equation
\begin{align}
	\label{0.1.3}
	f^n(z)+(f(z+c))^m=h(z),
\end{align}
where $c$ is a nonzero constant and $h$ is a given meromorphic function.  For example,
if $n=m=2$, $h=1$,  Liu et al.\cite{Lc} proved  \eqref{0.1.3} has an entire solution $f(z)=\sin(Az+B),$ where $B$ is a constant, $A=\frac{(4k+1)\pi}{2c}$, $k$ is an integer.
L\"u et al.\cite{lv} proved that if $f$ is a nonconstant meromorphic solution of \eqref{0.1.3} with finite order under the condition  $n\ge3,~m\ge2$ or $m\ge3,~n\ge2$ and $h(z)=e^{az+b}$, where $a,b$ are nonzero constants, then $m=n$ and $f(z)=ue^{\frac{az+b}{m}}$ with $u^m(1+e^{ac})=1$. Later, Bi et al. \cite{bi} cosidered the case $n=m=3$ and $h=e^{g(z)}$, where $g$ is a polynomial. Guo et al.\cite{Guo} extended this conclusion to the case  $h(z)=e^{g(z)}$, where $g$ is an entire function.

Since $e^{g(z)}$ is an entire function has no zero, then Laine et al.\cite{Laine} explore the existence problem of meromorphic
solutions to \eqref{0.1.3} when $h$  satisfies $N(r,h)+N(r,\frac{1}{h})=S(r,h)$, they obtained following result.

\vspace{6pt}
\begin{thmx}\cite[Theorem 2.2]{Laine}
	\label{thC}
	Let $h$ be meromorphic function  satisfying $N(r,h)+N(r,\frac{1}{h})=S(r,h)$ and let $n > m$ be positive integers. Then, \eqref{0.1.3} has no
	meromorphic solution $f$ with $\rho_2(f)<1$ except when $n = 2$ and $m = 1$. In this case,
	either:
	\begin{itemize}
		\item[\rm{(i)}]$f(z)=e^{az+b}-1$, where $a,b\in\mathbb C$ satisfy $e^{ac}=2$; or
		\item [\rm{(ii)}]$\overline{N}\left(r,\frac{1}{f'}\right)=T(r,f)+S(r,f)$, where $\overline{N}\left(r,\frac{1}{f'}\right)$ denotes the counting function corresponding to simple zeros of $f'$.
	\end{itemize}
\end{thmx}
\vspace{6pt}

Motivated by the fact that $h$ satisfies $N(r,h)+N(r,\frac{1}{h})=S(r,h)$ in  Theorem \ref{thC}, it is natural to ask the following question.

\vspace{6pt}
\begin{Que}
	\label{q1}
	What can we say about the meromorphic solutions of \eqref{0.1.3}  provided $h$ satisfies $N(r,h)=S(r,h)$ and $N(r,\frac{1}{h})=O(T(r,h))$ ?
\end{Que}
\vspace{6pt}

For Question \ref{q1}, Qi et al.\cite{Qxg} considered the solutions of  \eqref{0.1.3} when  $h(z)=p_1(z)e^{a_1(z)}+p_2(z)e^{a_2(z)}$ is a special type of function that satisfies $N(r,h)=S(r,h)$ and $N(r,\frac{1}{h})=O(T(r,h))$, where $k\ge1$ is an integer, $p_i$ are nonzero entire functions of order less than $k$, and  $a_i$ are nonzero polynomials, $i=1,2$. Upon careful inspection of Qi et al.\cite[Theorem 1.1]{Qxg}, the result of the theorem can be summarized and expressed as follows.

\vspace{6pt}
\begin{thmx}
	\label{thD}\cite[Theorem 1.1]{Qxg}
	Let $k\ge1$, $m$ and $n$  be  positive integers satisfy $n>m+2$. Suppose meromorphic function $f$ with $\rho_2(f)<1$ and $N(r,f)=S(r,f)$ is a  solution of   equation
	\begin{align}
		\label{0.1.4.1}
		f^n(z)+(f(z+c))^m=p_1(z)e^{a_1z^k}+p_2(z)e^{a_2z^k},
	\end{align}
	where $p_i$, are nonzero entire functions of order less than $k$, and  $a_i$ are nonzero constants, $a_1\not=a_2$, $i=1,2$. Then
	either:
	\begin{itemize}
		\item[\rm{(i)}]$f(z)=B_1(z)e^{\frac{a_1z^k}{n}}$, $B_1^n(z)=p_1(z)$  and $\frac{a_1m}{n}=a_2$ or
		\item[\rm{(ii)}]$f(z)=B_2(z)e^{\frac{a_2z^k}{n}}$, $B_2^m(z)=p_2(z)$ and $\frac{a_2m}{n}=a_1$.
	\end{itemize}
\end{thmx}
\vspace{6pt}

Later, Zhu et al.\cite{zx}  considered the meromorphic solution $f$ to Fermat type delay-differential equation
\begin{align}
	\label{0.1.5}
	f^n(z)+(f'(z+c))^m=p_1(z)e^{a_1z^k}+p_2(z)e^{a_2z^k},
\end{align}
	under the assumptions of Theorem \ref{thD}, and obtained the next result.

\vspace{6pt}
\begin{thmx}
	\label{the}\cite[Theorem 1.2]{zx}
	Under the assumptions of Theorem \ref{thD}. If equation \eqref{0.1.5} has a meromorphic function $f$ with $\rho_2(f)<1$ and $N(r,f)=S(r,f)$, then $f$ still satisfies the conclusions of Theorem \ref{thD}.
\end{thmx}
\vspace{6pt}

We find that the following equations have entire solutions, but these equations do not satisfy the condition  $n>m+2$ in Theorems \ref{thD} and \ref{the}.

\vspace{6pt}
\begin{exa}
	\label{e1}
	$f(z)=e^{iz}$ is a solution to $f^5(z)+(f{(z+2\pi)})^4=e^{5iz}+e^{4iz}$,  $f^5(z)+(f'{(z+2\pi)})^4=e^{5iz}+e^{4iz}$ and $f^4(z)+(f^{(2)}{(z+2\pi)})^5=e^{4iz}-e^{5iz}$, respectively.
\end{exa}
\vspace{6pt}

Therefore it is natural to ask the following question.

\vspace{6pt}
\begin{Que}
	\label{q22}
	Can we relax the condition $n>m+2$ in Theorems \ref{thD} and \ref{the} ?
\end{Que}
\vspace{6pt}

Firstly,  we give a positive answer to the  Question \ref{q22}.

\vspace{6pt}
\begin{thm}
	\label{th1.1}
	Let $l$, $k\ge1$, $m$ and $n$  be  integers satisfy $n>4$ and $m\ge2$ or $n=4$ and $m>4$. Suppose meromorphic function $f$ with $\rho_2(f)<1$ and $N(r,f)=S(r,f)$ is a  solution of   equation
	\begin{align}
		\label{0.1.4}
		f^n(z)+a(f^{(l)}(z+c))^m=p_1(z)e^{a_1z^k}+p_2(z)e^{a_2z^k},
	\end{align}
	where $p_i$, are nonzero entire functions of order less than $k$,  $a$ and  $a_i$ are nonzero constants, $a_1\not=a_2$, $i=1,2$.
	Then, either:
	\begin{itemize}
		\item[\rm{(i)}]$f(z)=B_1(z)e^{\frac{a_1z^k}{n}}$, $B_1^n(z)=p_1(z)$  and $\frac{a_1m}{n}=a_2$ or
		\item [\rm{(ii)}]$f(z)=B_2(z)e^{\frac{a_2z^k}{n}}$, $B_2^m(z)=p_2(z)$ and $\frac{a_2m}{n}=a_1$.
	\end{itemize}
\end{thm}

\begin{rem} Example \ref{e1} illustrates the existence of solutions to \eqref{0.1.4} of Theorem \ref{th1.1} when $n=5$ and $m=4$ or $n=4$ and $m=5$.   What's more,
	\begin{itemize}
		\item[(1)] if $m=1$ and $n>m+2=3$, then the form of the meromorphic solution  $f$ with $\rho_2(f)<1$ and $N(r,f)=S(r,f)$ to equations \eqref{0.1.4.1} and \eqref{0.1.5} was given by Mao et al.\cite[Theorem 1.1]{mzq2022}. If $m\ge 2$, Theorem \ref{th1.1} relaxes the condition $n>m+2$ in Theorems \ref{thD} and \ref{the} into $n>4$ and $m\ge2$.
		\item[(2)] Since  $h(z)=p_1(z)e^{a_1z^k}+p_2(z)e^{a_2z^k}$ is a special type of function that satisfies $N(r,h)=S(r,h)$ and $N(r,\frac{1}{h})=O(T(r,h))$, then Theorem \ref{th1.1} gives a partial solution  of Question \ref{q1}.
		
	\end{itemize}
	
\end{rem}

We provide an example to illustrate that the condition $n>4$ and $m\ge2$ or $n=4$ and $m>4$ of the Theorem \ref{th1.1} is necessary.
\vspace{6pt}
\begin{exa}
	\label{e2}
	$f(z)=e^{iz}+e^{-iz}$ is a solution of $f^4(z)-(f'(z+2\pi))^4=8e^{2iz}+8e^{-2iz}$, where $n=4$ and $f$ dosen't conform to the form specified in Theorem \ref{th1.1}.
\end{exa}

If the Fermat type delay-differential equation \eqref{0.1.4} degenerates into the Fermat type difference equation \eqref{0.1.4.1}, then  the following corollary is obtained by  using  a similar method as in the proof of Theorem \ref{th1.1}.

\vspace{6pt}
\begin{cor}
	\label{co2}
	Let $k\ge1$, $n>4$ and $m\ge2$ or $m>4$ and $n\ge2$ be  integers. Suppose meromorphic function $f$ with $\rho_2(f)<1$ and $N(r,f)=S(r,f)$ is a  solution of   equation \eqref{0.1.4.1}.
	Then,  $f$ still satisfies the conclusions of Theorem \ref{thD}.
\end{cor}
\vspace{6pt}

We provide the following example to illustrate the existence of the conclusion of the Corollary \ref{co2}.

\vspace{6pt}
\begin{exa}
	\label{e3}
	$f(z)=e^{iz}$  is a solution of $f^5(z)+f^2(z+2\pi)=e^{5iz}+e^{2iz}$, and $f^2(z)+f^5(z+2\pi)=e^{2iz}+e^{5iz}$, where $n=5$ and $m=2$ or $n=2$ and $m=5$.
\end{exa}
\vspace{6pt}

We find that if $n=m>4$, then \eqref{0.1.4}  does not have  meromorphic solution $f$ with $\rho_2(f)<1$ and $N(r,f)=S(r,f)$. Otherwise, by Theorem \ref{th1.1}, we have  $a_1=a_2$, which contradicts with the assumption $a_1\neq a_2$ in Theorem \ref{th1.1}. Thus, we derive the following corollary.

\vspace{6pt}
\begin{cor}
	\label{c1}
	Let $l$, $k\ge1$, $n>4$, be  integers, and let  $p_i$ be nonzero entire functions of order less than $k$, $a$ and  $a_i$ be nonzero constants with $a_1\not=a_2$ for $i=1,2$. Then  equation
	\begin{align*}
		f^n(z)+a(f^{(l)}(z+c))^n=p_1(z)e^{a_1z^k}+p_2(z)e^{a_2z^k}
	\end{align*}
	has no meromorphic solution $f$ such that $\rho_2(f)<1$ and $N(r,f)=S(r,f)$.
\end{cor}

\vspace{6pt}
\begin{rem}	\label{rm3}
	If $f$ is an entire function and $c=0$, then  Theorem \ref{th1.1} and  Corollary \ref{c1} still hold by using the similar method, without the condition $\rho_2(f)<1$. Therefore, Corollary \ref{c1} provides a positive answer to the following conjecture proposed by Gao et al.\cite{Gao}.
	
	\begin{con}\cite[Conjecture 2]{Gao}
		\label{con1}
		~~Let $l$ and $n$ be positive integers with $n > 4$, and let $p_i$, $a$, and $a_i$ be nonzero constants with $a_1 \neq a_2$ for $i = 1, 2$. Then the equation
		\begin{align*}
			f^n(z)+a(f^{(l)}(z))^n=p_1e^{a_1z}+p_2e^{a_2z}
		\end{align*}
		has no   entire solution.
	\end{con}
\end{rem}

\vspace{6pt}
Example \ref{e2} demonstrates that the condition $n>4$ in the Corollary \ref{c1} cannot be further weakened. Therefore,
the following question is arisen.
\vspace{6pt}

\begin{Que}
	\label{q2}
	Let $l$, $k\ge1$, $n\le4$, be  integers, and let  $p_i$ be nonzero entire functions of order less than $k$, and $a$, $a_i$ be nonzero constants with $a_1\not=a_2$, $i=1,2$. Then, what  happen to the solutions of the Fermat type delay-differential equation
	\begin{align*}
		f^n(z)+a(f^{(l)}(z+c))^n=p_1(z)e^{a_1z^k}+p_2(z)e^{a_2z^k}?
	\end{align*}
\end{Que}

Gundersen et al.\cite{Gundersen} obtained some partial results for the case $n=2$, which are also related to the trigonometric identity $(\cos z)^2-(\sin z)^2=\cos 2z$.

\vspace{6pt}
\begin{thmx}\cite{Gundersen}
	\label{Thc}
	The only entire solutions of the differential equation
	\begin{align}
		\label{1.2}
		f^2(z)-(f'(z))^2=\frac{1}{2}e^{i2z}+\frac{1}{2}e^{-2iz}
	\end{align}
	are the four solutions $f=\cos z,~i\sin z$.
\end{thmx}
\vspace{6pt}

For the case $n=4$, Gao et al. \cite{Gao}  obtained  the following result, which is connected to  $(\cos z)^4-(\sin z)^4=\cos2 z$.

\vspace{6pt}
\begin{thmx}\cite{Gao}
	\label{thd}
	The binomial differential equation
	\begin{align*}
		f^4(z)-(f'(z))^4=\frac{1}{2}e^{i2z}+\frac{1}{2}e^{-2iz}
	\end{align*}
	has the eight entire solutions $f_j(z)=e^{\frac{2j\pi i}{4}}\cos z$ and $f_k=e^{\frac{(2k+1))\pi i}{4}}\sin z$ for $j = 0, 1, 2, 3 $ and $ k = 0, 1, 2, 3.$
\end{thmx}

\vspace{6pt}
Based on Theorem \ref{Thc},   Gao et al.\cite{Gao}  considered a more general Fermat type differential equation.
\vspace{6pt}

\begin{thmx}\cite{Gao}\label{thE}
	Let $a, b, p_1, p_2,$  be nonzero constants satisfying $9ab^2\not=-4$.
	Then, the equation
	\begin{align}
		\label{f}
		f^2(z)+a(f'(z))^2=p_1e^{bz}+p_2e^{-bz}
	\end{align}
	has entire solutions $f$ if and only if the condition
	$ab^2+1=0$ or $ab^2-4=0$
	holds. Moreover,
	\begin{itemize}
		\item[\rm{(i)}] If $ab^2+1=0$, then  $f=t_{1_i}e^{bz}+t_{2_i}e^{-bz}+r_i$, where $r_i$ are the four roots of $r^4=-p_1p_2$, $2t_{1_i}r_i=p_1$, $2t_{2_i}r_i=p_2$, $i=1,2,3,4$.
		\item[\rm{(ii)}] If $ab^2-4=0$, then $f=l_{1_i}e^{\frac{bz}{2}}+l_{2_i}e^{\frac{-bz}{2}}$,  where $l_{1_i}$ are the square roots of $\frac{p_1}{2}$ and $l_{2_i}$ are the square roots of $\frac{p_2}{2}$, $i=1,2$.
	\end{itemize}
\end{thmx}

For Theorem \ref{thE}, if $9ab^2=-4$, Gao et al. \cite{Gao} found that $f(z)=e^z+e^{-3z}$ is a solution of $f^2(z)-\frac{1}{9}(f'(z))^2=\frac{8}{9}e^{2z}+\frac{8}{3}e^{-2z}$. Therefore  they posed the following conjecture.

\vspace{6pt}
\begin{con}\cite[Conjecture 1]{Gao}
	\label{con2}
	~~If $9ab^2=-4$, then  entire solutions of \eqref{f} are  $f(z)=l_1e^{\frac{-3}{2}bz}+l_2e^{\frac{1}{2}bz}$, where $l_2^2=\frac{9}{8}p_1$, $l_1^2=\frac{p_2^2}{8p_1}$ or
	$f(z)=l_3e^{\frac{3}{2}bz}+l_4e^{\frac{-1}{2}bz}$, where $l_4^2=\frac{9}{8}p_2$, $l_3^2=\frac{p_1^2}{8p_2}$.
\end{con}
\vspace{6pt}

Next, we consider the Question \ref{q2} and Conjecture \ref{con2},
and  obtain the following result.

\begin{thm}
	\label{th1.2}
	Let $f$ be an entire solution of equation
	\begin{align}
		\label{1.6} f^n(z)+a(f'(z))^n=p_1e^{bz}+p_2e^{-bz},
	\end{align}
	where $2\le n \le 4$ is an integer, $a,p_1,p_2,b$ are nonzero constants. Then $n=2,4$, and one the following
	asserts holds.
	\begin{enumerate}
		\item[\rm{(i)}]  When $n=4$, then $f(z)=k_1e^{\frac{b}{2}z}+k_2e^{\frac{-b}{2}z}$, $a(\frac{b}{2})^4=-1$, where $8(k_1)^3k_2=p_1$, $8(k_2)^3k_1=p_2$.
		
		\item[\rm{(ii)}]  When $n=2$, one the following asserts holds.
		
		(1)If $9ab^2\not=-4$, then   Theorem \ref{thE} holds.
		
		(2)If $9ab^2=-4$, then
		$f(z)=l_1e^{\frac{-3}{2}bz}+l_2e^{\frac{1}{2}bz}$, where $l_2^2=\frac{9}{8}p_1$, $l_1^2=\frac{p_2^2}{8p_1}$ or
		$f(z)=l_3e^{\frac{3}{2}bz}+l_4e^{\frac{-1}{2}bz}$, where $l_4^2=\frac{9}{8}p_2$, $l_3^2=\frac{p_1^2}{8p_2}$.		
	\end{enumerate}
\end{thm}

\vspace{6pt}
\begin{rem}
	Theorem \ref{th1.2}-(ii) provides a positive answer to the Conjecture \ref{con2}.  What's more,
	when $n=2,a=-1$, $p_1=p_2=\frac{1}{2}$, $b=2i$,
	Theorem  \ref{th1.2} reduces to Theorem \ref{Thc}, when $n=4,a=-1$, $p_1=p_2=\frac{1}{2}$, $b=2i$,
	Theorem  \ref{th1.2} reduces to Theorem \ref{thd}.
\end{rem}

\vspace{6pt}
Finally, we give two  examples for (i) and (ii) of Theorem \ref{th1.2}, respectively.
\vspace{6pt}

\begin{exa}
	$f(z)=e^{iz}+e^{-iz}$ is a solution of $f^4(z)-(f'(z))^4=8e^{2iz}+8e^{-2iz}$.
\end{exa}

\begin{exa}\cite[Remark 2]{Gao}
	The equation $f^2(z)-\frac{1}{9}(f'(z))^2=\frac{8}{9}e^{2z}+\frac{8}{3}e^{-2z}$ has a solution $f(z)=e^z+e^{-3z}$, where $n=2,a=\frac{-1}{9},b=2$.
\end{exa}

\section{Some Lemmas}\label{sec3}
In this section, we will collect some preliminary results for proving our results. For a nonconstant meromorphic function $f$ and positive integer $p$, let $n_p{(r,\frac{1}{f})}$ denotes the number of zeros of $f$ in $\lvert{z}\lvert  \le {r}$, counted in the following manner: a zero of $f$ of multiplicity $l$ is
counted exactly  min$\{l,p\}$ times, its corresponding integrated counting function is denoted by $N_p(r,\frac{1}{f})$.

The first lemma is the simple form of Cartan's second main theorem.

\vspace{6pt}
\begin{lem}\cite{Cartan,Gundersen2004}
	\label{lm1.1}
	Let $f_1,\,f_2,\dots,f_p$ be linearly independent entire functions, where $p\ge2$. Assume that for each complex number $z$, $\max\{{\lvert{f_1(z)}\rvert},\dots,{\lvert{f_p(z)}\rvert}\}{>0}$. For $r>0$, set
	\begin{align*}
		T(r)={\frac{1}{2\pi}}{\int_{0}^{2\pi}}u(re^{i\theta})d\theta-u(0), \, u(z)=\sup\limits_{1\le{j}\le{p}}\{\log {\lvert{f_j(z)}\rvert}\}.
	\end{align*}
	Set $f_{p+1}=f_1+\cdots+f_p$. Then,
	\begin{align*}
		T(r){\le}\sum_{j=1}^{p+1}{N_{p-1}\left(r,\frac{1}{f_j}\right)}+S(r){\le}(p-1)\sum_{j=1}^{p+1}{\overline{N}_{p-1}\left(r,\frac{1}{f_j}\right)}+S(r),
	\end{align*}
	where $S(r)=O\,(\log{T(r))}+O\,(\log{r})$, $r\rightarrow\infty,\,r\notin{E}$. If there exists at least one of the quotients $\frac{f_j}{f_m}$ $(j,\,m\in\{1,\ldots,p+1\},\,j\not=m)$ is a transcendental function, then $S(r)=o\,(T(r)),$ $r\rightarrow\infty,\,r\notin{E}$, where E is the set of finite linear measure.
\end{lem}
\vspace{6pt}

\begin{lem}\cite{Cartan}
	\label{lm1.2}
	Suppose that the hypothesis of Lemma \ref{lm1.1} hold. Then for any j and m $(j,\,m\in\{1,\ldots,p+1\},\,j\not=m)$,
	\begin{center}
		$T\left(r,\frac{f_j}{f_m}\right)=T(r)+O\,(1)$,
	\end{center}
	and for any  $j\in\{1,\ldots,p+1\}$,
	\begin{center}
		$N\left(r,\frac{1}{f_j}\right)\le T(r)+O\,(1)$.
	\end{center}
\end{lem}
\vspace{6pt}

\begin{lem}\cite{halburd2014}
	\label{lm 1.3}
	Let $f$ be a nonconstant meromorphic function of hyper-order $\rho_2(f)<1$, $k$ be a positive integer and $c$, $h$ be  nonzero complex numbers. Then the following statements hold.
	\begin{itemize}
		\item[\rm{(i)}]$m\left(r,\frac{f^{(k)}(z+c)}{f(z)}\right)=S(r,f)$.
		\item[\rm{(ii)}]$N\left(r,\frac{1}{f(z+c)})\right)=N\left(r,\frac{1}{f(z)}\right)+S(r,f)$, $N(r,{f(z+c)})=N(r,{f(z)})+S(r,f).$
		\item[\rm{(iii)}]$T(r,f(z+c))=T(r,f(z))+S(r,f)$.
	\end{itemize}
\end{lem}

The next lemma is Borel type theorem, which can be found in \cite{yang2003}.
\begin{lem}\cite[Theorem 1.51]{yang2003}
	\label{lm3.1}
	Let $f_1,\,f_2,\ldots,f_n$ be meromorphic functions, $g_1,\,$
	$g_2,\ldots,g_n$ be entire functions satisfying
	the following conditions,
	\begin{itemize}
		\item [\rm{(i)}]$\sum\limits_{j=1}^n{{f_j(z)}}e^{g_j(z)}\equiv0$,
		\item[\rm{(ii)}]For $1\le{j}<{k}\le{n}$, $g_j-g_k$ is not constant,
		\item[\rm{(iii)}]For $1\le{j}\le{n},1\le{t}<{k}\le{n}$, $T(r,f_j)=o\,\{T(r,e^{g_t-g_k})\},$\, $r\rightarrow\infty,\,r\notin{E}$, where E is the set of finite linear measure.
	\end{itemize}
	Then $f_j(z)\equiv0,\, j=1,\ldots,n$.
\end{lem}
\vspace{6pt}

The next result is concerned asymptotic estimation of exponential polynomials.
\vspace{6pt}

\begin{lem}\cite[Lemma 2.5.]{mzq2022}
	\label{lm3.2}
	Let $m$, $q$ be positive integers, $\omega_1,\ldots,\omega_m$ be distinct nonzero complex numbers, and $H_0,\,H_1,$ $\ldots,H_m$ be meromorphic functions of order less than $q$ such that $H_j\not\equiv 0, \, 1\le{j}\le{m}$. Set $\varphi(z)=H_0+\sum\limits_{j=1}^m{{H_j(z)}}e^{\omega_j{z^q}}$. Then the following statements hold.
	\begin{itemize}
		\item[\rm{(i)}]There exist two positive numbers $d_1<d_2$, such that for sufficiently large $r$,   $d_1r^q{\le}T(r,\varphi){\le}d_2r^q$.
		\item[\rm{(ii)}]If $H_0\not\equiv0$, then $m\left(r,\frac{1}{\varphi}\right)=o\,(r^q)$.
	\end{itemize}
\end{lem}

The following  lemma  is concerned with the Fermat type differential equation due to   Yang et al. \cite{yang2004}.

%\begin{lem}\cite[Theorem 1]{Yang1970}\label{lm2.5}
%%%%%%%%\end{lem}

\begin{lem}\cite[Theorem 2] {yang2004}
	\label{lm3.3}
	Let $b$ be a nonzero constant, and $a(z)$ be a meromorphic function. If $a(z)$ is not constant, then
	\begin{align*}
		f^2(z)+b(f')^2=a(z)
	\end{align*}
	has no transcendental meromorphic solution $f$ such that $T(r,a)=S(r,f)$.
\end{lem}

\vspace{6pt}
In the case of meromorphic functions of finite order, the following
logarithmic derivative lemma can be found in \cite{Cherry}.
\vspace{6pt}

\begin{lem}\cite[Theorem 3.5.1]{Cherry}\label{lm3.4}
	Let $p$ be a nonnegative real number, $f$ be a nonconstant meromorphic function of finite order $p$. Then  for all $\varepsilon>0$, $r$ sufficiently large,
	\begin{align*}
		m(r,\frac{f'}{f})\le\max\{0,(p-1+\varepsilon)\log r\}+O(1).
	\end{align*}
\end{lem}

The following lemma concerns the  lower order of composite functions.
\vspace{6pt}

\begin{lem}\cite[Lemma 1]{be}\label{lm3.5}
Let $f$ be a meromorphic function and let $g$ be a transcendental entire function. If the lower
order $\mu (f\circ g)<\infty$, then $\mu (f)=0$.
\end{lem}

%\begin{lem}\cite{Gundersen1,hayman1} \label{lm3.5}
%For an integer $n>6$, there do not exist transcendental entire solution  $f,g,h$ satisfying \begin{align*}
	%f^n(z)+g^n(z)+h^n(z)=1.
	%\end{align*}
	%\end{lem}

\vspace{6pt}	
	The next result is concerned asymptotic estimation of the solution $f$ of \eqref{0.1.4} in Theorem \ref{th1.1}.
\vspace{6pt}	
	
	\begin{lem}
		\label{lm3.6}
		Let $l$, $k\ge1$, $n>4$ and $m\ge2$ or $m>4$ and $n\ge2$ be  integers, and let $f$ be a meromorphic solution of \eqref{0.1.4} with $N(r,f)=S(r,f)$ and $\rho_2(f)<1$. Then
		$\mu(f)=k$.
	\end{lem}
	\begin{proof}[Proof of Lemma \ref{lm3.6}]
		Let $f$ be a meromorphic solution of \eqref{0.1.4} with $N(r,f)=S(r,f)$ and $\rho_2(f)<1$. By Lemmas \ref{lm3.2} and \ref{lm 1.3}, for sufficiently large $r$, we get $d_1r^k= T(p_1(z)e^{a_1z^k}+p_2(z)e^{a_2z^k}) \le (n+m) T(r,f) +S(r,f)$, where $d_1$ is a nonzero constant. Therefore $\mu(f)\ge k$.
		
		 We claim  $\mu(f)=k$.  Otherwise $\mu(f)>k$, which means   $p_1(z)e^{a_1z^k}+p_2(z)e^{a_2z^k}$ is a small function of $f$ and $f^{(l)}(z+c)$. Since $n>4$ and $m\ge2$ or  $n\ge2$ and $m>4$,  ${\Theta(\infty,f)}={\Theta(\infty,f^{(l)}(z+c))}=1$, by Theorem \ref{thb}, we have \eqref{0.1.4} has no nonconstant meromorphic solution, which  implies a contradiction. Thus $\mu(f)=k$.
	\end{proof}
	
	\section{Proof of Theorem \ref{th1.1} and Corollary \ref{co2}}
	
	\begin{proof}[Proof of Theorem \ref{th1.1}] By  Lemma \ref{lm3.6}, we get $\mu(f)=k$.
		 The following proof will be divided into two cases: $a(f^{(l)}(z+c))^m$, $p_1(z)e^{a_1z^k}$ and $p_2(z)e^{a_2z^k}$ are linearly dependent or linearly independent.
		
		\setcounter{case}{0}
		\begin{case}
			\rm{ $a(f^{(l)}(z+c))^m$, $p_1(z)e^{a_1z^k}$ and $p_2(z)e^{a_2z^k}$ are linearly dependent. There exist constants  $b_1, b_2$, not all zero, such that
				\begin{align}
					\label{3.1}
					a(f^{(l)}(z+c))^m=b_1p_1(z)e^{a_1z^k}+b_2p_2(z)e^{a_2z^k}.
				\end{align}
				Suppose that neither $b_1$ nor $b_2$ is zero. Rewriting \eqref{3.1} into
				\begin{align}
					\label{3.2}
					\frac{a}{b_1p_1}\left(\frac{f^{(l)}(z+c)}{e^{\frac{a_1z^k}{m}}}\right)^m+ \frac{-b_2p_2}{b_1p_1}\left( e^{\frac{(a_2-a_1)}{4}z^k}\right)^4=1.
				\end{align}
				It is easy to see $T\left(r,\frac{f^{(l)}(z+c)}{e^{\frac{a_1z^k}{m}}}\right)=O(r^k)$, since $n>4$ and $m\ge2$ or  $n=4$ and $m>4$,  ${\Theta(\infty,f)}={\Theta(\infty,f^{(l)}(z+c))}=1$, by Theorem \ref{thb}, then \eqref{3.2} has no nonconstant meromorphic solution, which  implies a contradiction. Therefore  $b_1=0$ or $b_2=0$.
				
				If $b_1=0$, then $a(f^{(l)}(z+c))^m=b_2p_2(z)e^{a_2z^k}$. Substituting it into the equation \eqref{0.1.4}, we obtain
				\begin{align}
					\label{3.3}
					f^n(z)=p_1(z)e^{a_1z^k}+(1-b_2)p_2(z)e^{a_2z^k}.
				\end{align}
				Using a similar method as above for equation \eqref{3.2} to analyze \eqref{3.3}, we have $b_2=1$. Then $f^n(z)=p_1(z)e^{a_1z^k}$, which means $f(z)=B_1(z)e^{\frac{a_1z^k}{n}}$, where $B_1^n=p_1$.  Substituting it into the equation \eqref{0.1.4}, we obtain $ma_1=na_2$. We get the  conclusion (i) of Theorem \ref{th1.1}.
				
				If $b_2=0$, then $a(f^{(l)}(z+c))^m=b_1p_1(z)e^{a_1z^k}$. Substituting it into the equation \eqref{0.1.4},  and using the similar method as above for $b_1=0$ to analyze $b_2=0$, we have $f(z)=B_2(z)e^{\frac{a_2z^k}{n}}$, where $B_2^n=p_2$ and $na_1=ma_2$. We get conclusion (ii) of Theorem \ref{th1.1}.
				
			}
		\end{case}

		\begin{case}
			\rm{$a(f^{(l)}(z+c))^m$, $p_1(z)e^{a_1z^k}$ and $p_2(z)e^{a_2z^k}$ are linearly independent. Since $N(r,f)=S(r,f)$, $p_1, p_2$ are small function of $e^{a_1z^k}$, then it is easy  to construct a meromorphic function $g(z)$ satisfying $T(r,g)=S(r,f)$ such that $gf^n$, $ga(f^{(l)}(z+c))^m$, $gp_1(z)e^{a_1z^k}$ and $gp_2(z)e^{a_2z^k}$ become entire functions and  have no common zeros.  Now
				\begin{align}
					\label{3.4}
					gf^n(z)=-ga(f^{(l)}(z+c))^m+gp_1(z)e^{a_1z^k}+gp_2(z)e^{a_2z^k},
				\end{align} and $\max\{{\lvert{gf^n(z)|,|ga(f^{(l)}(z+c))^m}\rvert},{\lvert{gp_1(z)e^{a_1z^k}}\rvert},{\lvert gp_2(z)e^{a_2z^k}\rvert}\}{>0}$.
				Using Lemmas \ref{lm1.1}, \ref{lm1.2} and \ref{lm 1.3} to \eqref{3.4}, we have
				\begin{align}
					\label{3.5}
					nN\left(r,\frac{1}{f}\right)&= N(r,\frac{1}{f^n}){\le}N(r,\frac{1}{f^ng})+S(r,f)	{\le}T(r)+S(r,f)\\ \nonumber
					&{\le}\sum_{j=1}^2N_2(r,\frac{1}{p_je^{a_jz^k}g})+N_2(r,\frac{1}{-ga(f^{(l)}(z+c))^m})+ N_2(r,\frac{1}{f(z)^ng})\\ \nonumber
					&+o(T(r))+S(r,f)\\ \nonumber
					&{\le}2N\left(r,\frac{1}{f^{(l)}(z+c)}\right)+2N(r,\frac{1}{f})+o(T(r))+S(r,f)\\ \nonumber
					&{\le}4N\left(r,\frac{1}{f}\right)+o(T(r))+S(r,f),\nonumber
				\end{align}
				where\begin{align*}
					T(r)={\frac{1}{2\pi}}{\int_{0}^{2\pi}}u(re^{i\theta})d\theta-u(0),
				\end{align*}
				 and
				$u(z)=\sup\{ \log{\lvert{ga(f^{(l)}(z+c))^m}\rvert},\log{\lvert{gp_1(z)e^{a_1z^k}}\rvert}, \log{\lvert{gp_2(z)e^{a_2z^k}}\rvert}  \}.$
				
			By \eqref{3.5},	it is easy to get $T(r)\le O(T(r,f))$, Thus $o(T(r))=S(r,f)$.
				Note that \eqref{3.5} holds when $n>4$ and $m\ge2$ or $n=4$ and $m>4$. Therefore, by \eqref{3.5},  we have
				
				\begin{align}
					\label{3.6}
					2N(r,\frac{1}{f})\le(n-2)N(r,\frac{1}{f})\le2N\left(r,\frac{1}{f^{(l)}(z+c)}\right)+S(r,f).
				\end{align}
				
				We first consider the case $n>4$ and $m\ge2$. By \eqref{3.5}, we have  $(n-4)N\left(r,\frac{1}{f}\right)\le S(r,f)$. 	Since $n>4$, then $N(r,\frac{1}{f})=S(r,f)$. Because $N(r,f)=S(r,f)$ and $N\left(r,\frac{1}{f}\right)=S(r,f)$,  by the  Weierstrass's factorization theorem, then   $f(z)=\gamma (z)e^{h(z)}$,  where $\gamma$ is a small function of $f$, $h$ is a nonconstant entire function. Thus $T(r,f)=O(T(r,e^{h}))$, $\mu (f)= \mu (e^{h})=k$, and $\rho(f)=\rho(e^h)$. If $h$ is  transcendental, since $\mu (e^z)=1$, by Lemma \ref{lm3.5}, then $\mu(f)=\mu(e^{h})=\infty$,   which is contradictory to  $\mu(f)=k$.  If $h$ is a polynomial, then $\mu(e^h)=\rho(e^h)=k$. Therefore $T(r,f)=O(T(r,e^h))=O(r^k)$. Now, it can be derived from the definition of $T(r)$ that $T(r)=O(r^k)=O(T(r,f))$.  By \eqref{3.5} and $N\left(r,\frac{1}{f}\right)=S(r,f)$, we have  $T(r)\le 4N\left(r,\frac{1}{f}\right)+S(r,f)=S(r,f)$, which is a contradiction.

				Next, we consider the case   $n=4$ and $m>4$.
				If $f^4(z)$, $p_1(z)e^{a_1z^k}$ and $p_2(z)e^{a_2z^k}$ are linearly dependent, then there exist constants  $b_3, b_4$, not all zero, such that
				$f^4(z)=b_3p_1(z)e^{a_1z^k}+b_4p_2(z)e^{a_2z^k}$. Which means $a(f^{(l)}{z+c})^m=(1-b_3)p_1(z)e^{a_1z^k}+(1-b_4)p_2(z)e^{a_2z^k}.$  Therefore $a(f^{(l)}(z+c))^m$, $p_1(z)e^{a_1z^k}$ and $p_2(z)e^{a_2z^k}$ are linearly dependent, this contradicts the assumption of Case 2.
				Thus $f^4(z)$, $p_1(z)e^{a_1z^k}$ and $p_2(z)e^{a_2z^k}$ are linearly independent, by using the similar method as above for  $a(f^{(l)}(z+c))^m$, $p_1(z)e^{a_1z^k}$ and $p_2(z)e^{a_2z^k}$ are linearly independent, we get
				\begin{align}
					\label{3.7}
					mN\left(r,\frac{1}{f^{(l)}(z+c)}\right)&= N(r,\frac{1}{(f^{(l)}(z+c))^m})\\ \nonumber
					&{\le}N(r,\frac{1}{ga(f^{(l)}(z+c))^m})+S(r,f)\\ \nonumber
					&{\le}T_1(r)+S(r,f)\\ \nonumber
					&{\le}\sum_{j=1}^2N_2(r,\frac{1}{p_je^{a_jz^k}g})+N_2(r,\frac{1}{ga(f^{(l)}(z+c))^m})+ N_2(r,\frac{1}{-f^4(z)g})\\  \nonumber
					&+o(T_1(r))+S(r,f)\\ \nonumber
					&{\le}2N\left(r,\frac{1}{f^{(l)}(z+c)}\right)+2N(r,\frac{1}{f})+o(T_1(r))+S(r,f), \nonumber
				\end{align}
				where
				\begin{align*}
					T_1(r)={\frac{1}{2\pi}}{\int_{0}^{2\pi}}u_1(re^{i\theta})d\theta-u_1(0),
				\end{align*}
				$u_1(z)=\sup\{ \log{\lvert{gf^4(z)}\rvert},\log{\lvert{gp_1(z)e^{a_1z^k}}\rvert}, \log{\lvert{gp_2(z)e^{a_2z^k}}\rvert}  \}.$
				By \eqref{3.7}, then $T_1(r)\le O(T(r,f))$,  $o(T_1(r))=S(r,f)$.
				From \eqref{3.6} and \eqref{3.7},  then $T_1(r)\le4N\left(r,\frac{1}{f^{(l)}(z+c)}\right)+S(r,f)$, $mN\left(r,\frac{1}{f^{(l)}(z+c)}\right)$$\le 4N\left(r,\frac{1}{f^{(l)}(z+c)}\right)+S(r,f)$.
				Since $m>4$, just as in case $n>4$ and $m\ge2$,  then we can aslo have a contradiction.
				
			}
		\end{case}
	\end{proof}
	
	\begin{proof}[Proof of Corollary \ref{co2}]
		The proof Corollary \ref{co2} is largely similar to that of Theorem \ref{th1.1}, so we will only elaborate on the differences.
		The following proof will be divided into two cases:   $(f^{(l)}(z+c))^m$, $p_1(z)e^{a_1z^k}$ and $p_2(z)e^{a_2z^k}$ are linearly dependent or linearly independent.
		
		$(f^{(l)}(z+c))^m$, $p_1(z)e^{a_1z^k}$ and $p_2(z)e^{a_2z^k}$ are linearly dependent. Then the proof method is exactly the same as that in Case 1 of Theorem \ref{th1.1}, so it is omitted here.
		
		$(f^{(l)}(z+c))^m$, $p_1(z)e^{a_1z^k}$ and $p_2(z)e^{a_2z^k}$ are linearly independent.
		If $n>4$ and $m\ge2$,  similarly to Case 2 of Theorem \ref{th1.1}, we have \eqref{3.5}. By \eqref{3.5}, we have  $(n-4)N\left(r,\frac{1}{f}\right)\le S(r,f)$, and $T(r)\le 4N\left(r,\frac{1}{f}\right)+S(r,f) $. Since $n>4$, then $N(r,\frac{1}{f})=S(r,f)$, we have $T(r)=S(r,f)$, which is a contradiction.
		
		If $n\ge2$ and $m>4$, similarly to \eqref{3.7} in Case 2 of Theorem \ref{th1.1},   we get $T_1(r)\le4N\left(r,\frac{1}{f^(z+c)}\right)+S(r,f)$ and $mN\left(r,\frac{1}{f(z+c)}\right)\le 4N\left(r,\frac{1}{f}\right)+S(r,f)$.  By Lemma \ref{lm 1.3}, we have $mN(r,\frac{1}{f})+S(r,f)=mN\left(r,\frac{1}{f(z+c)}\right)\le 4N\left(r,\frac{1}{f}\right)+S(r,f)$. Since $m>4$, then $N(r,\frac{1}{f})=S(r,f)$, we have $T_1(r)=S(r,f)$, which is a contradiction.
	\end{proof}
	
	\section{Proof of  Theorem \ref{th1.2}}
	Let $f$ be an entire function of \eqref{1.6},
	 by Lemma \ref{lm3.2} and \eqref{1.6}, we obtain $O(r)=m(r,p_1e^{bz}+p_2e^{-bz})\le 3nT(r,f)$. Thus, $\mu(f)\ge 1$.  If $\mu(f)>1$, then $p_1e^{bz}+p_2e^{-bz}$ is a small function of $f$.  Using Theorem \ref{thb} and Lemma \ref{lm3.3} to \eqref{1.6}, then \eqref{1.6} has no  transcendental meromorphic solution $f$, which is  a contradiction. Therefore, $\mu(f)=1$.
	
	We claim $f$ have infinite many zeros, otherwise, it is easy to get a contradiction by  Hadamard factorization theorem and   Lemma \ref{lm3.1}. Because $p_1e^{az}+p_2e^{-az}$ has only simple zero, then \eqref{1.6} implies $f'$ and $f$  don't have common zero, which means the zero of $f$ and $f'$ is simple.
	
	Differentiating \eqref{1.6} , then
	\begin{align}
		\label{4.7}
		nf^{n-1}f'+ na(f')^{n-1}f''=p_1be^{bz}-p_2be^{-bz}.
	\end{align}
If $f'$ has zeros, then
from \eqref{4.7} and $n\ge2$,  the zero of $f'$ must be the zero of $p_1e^{bz}+p_2e^{-bz}$. Therefore $N(r,\frac{1}{f'})=O(r)$, by the  Weierstrass's factorization theorem, then   $f'(z)=\gamma (z)e^{h(z)}$,  where $\gamma$ is an entire function with $\mu(\gamma)=\rho(\gamma)=1$, $h$ is a  entire function. If $h$ is  transcendental, since $\mu (e^z)=1$, by Lemma \ref{lm3.5}, then $\mu(f)=\mu(e^{h})=\infty$,  which is contradictory to  $\mu(f')=1$.  If $h$ is a polynomial, then we claim $\deg(h)\le 1$, otherwise $\mu(f')=\deg(h)>1$,  which is contradictory to  $\mu(f)=1$.  Therefore $\mu(f')=\rho(f')=1$, which means $T(r,f)=O(r)$. If $f'$ has no zeros, by applying the same method, then $\rho (f)=\mu (f)=1$,  which means $T(r,f)=O(r)$. Therefore, regardless of whether $f'$ has any zeros, it always holds that $T(r,f)=O(r)$.

Differentiating \eqref{4.7} two times, then
	\begin{align}
		\label{4.8}
		A_n^2f^{n-2}(f')^2+nf^{n-1}f''+A_n^2a(f')^{n-2}(f'')^2+na(f')^{n-1}f'''=b^2p_1e^{bz}+b^2p_2e^{-bz},
	\end{align}
	\begin{align}
		\label{4.9}
		A_n^3f^{n-3}(f')^3+3A_n^2f^{n-2}f'f''+nf^{n-1}f'''+A_n^3a(f')^{n-3}(f'')^3\\ \nonumber
		+3aA_n^2(f')^{n-2}f''f''' +na(f')^{n-1}f^{(4)}=b^3p_1e^{bz}-b^3p_2e^{-bz},
	\end{align}
	where   $A_n^m=\frac{n!}{(n-m)!}$.
	Using \eqref{4.8} and \eqref{1.6} to eliminate $e^{bz}$ and $e^{-bz}$, then
	\begin{align}
		\label{4.10}
		f^{n-2}[b^2f^2-A_n^2(f')^2-nff'']=(f')^{n-2}[-b^2a(f')^2+A_n^2a(f'')^2+naf'f'''].
	\end{align}
	
	Next we will divide the proof into the following three cases: $n=4$, $n=3$ and $n=2$.
	
	\vspace{6pt}
	\setcounter{case}{0}
	\begin{case}
		\rm{
			$n=4$. \eqref{4.9} and \eqref{4.7} implies
			\begin{align}
				\label{4.11}
				4f^{3}f'''=4b^2f^{3}f'+4b^2a(f')^{3}f''-A_4^3f(f')^3-\\ \nonumber
				3A_4^2f^{2}f'f''-aA_4^3(f')(f'')^3-3aA_4^2(f')^{2}f''f'''-4a(f')^{3}f^{(4)}.
			\end{align}
			Then  \eqref{4.11} implies the zero of $f'$ must be the zero of $f'''$. Therefore
			\begin{align}
				\label{4.12}
				\frac{f'''}{f'}=g,
			\end{align}
			where $g$ is a nonzero entire function. Noting  that $T(r,f)=O(r)$  and  Lemma \ref{lm3.4},  then $m(r,g)=o(\log r)$, which means $g$ is a nonzero constant. By integration \eqref{4.12},
			\begin{align}
				\label{4.121}
				gf+h=f'',
			\end{align}
			where $h$ is a constant.
			
			In the following, we discuss whether $h=0$.
			
			\vspace{6pt}
			\setcounter{subsection}{1}
			\setcounter{subcase}{0}
			\renewcommand{\thesubcase}{\arabic{subsection}.\arabic{subcase}}
			\begin{subcase}
				\rm{
					$h\not=0$. Let $z_0$ be a zero of $f$, substituting  $z_0$, \eqref{4.12} into \eqref{4.10} and \eqref{4.11}, get
					\begin{align}
						\label{4.13}
						(-b^2a+4ag)(f'(z_0))^2&=-A_4^2a(f''(z_0))^2,\\\nonumber
						[b^2-10g](f'(z_0))^2&=6(f''(z_0))^2.
					\end{align}
					\eqref{4.13} implies $g=\frac{b^2}{16}$. Therefore rewriting \eqref{4.121} into
					\begin{align}
						\label{4.14}
						f''-\frac{b^2}{16}f-h=0.
					\end{align}
					Solving the \eqref{4.14}, we obtain  $f(z)=t_1e^{\frac{b}{4}z}+t_2e^{\frac{-b}{4}z}+t_3$, where $t_i$ are constants, $i=1,2,3$, and $t_3=-\frac{16}{b^2}h\not=0$. Substituting  $f$ into \eqref{1.6} and using Lemma \ref{lm3.1}, then we get $t_3=0$, which is a contradiction.
				}
			\end{subcase}

			\vspace{6pt}
			\begin{subcase}
				\rm{
					$h=0$. Then from \eqref{4.121}, $f'''=gf'$. Let $z_0$ is a zero of $f$, substituting $z_0$ and $f'''$ into \eqref{4.10}, then from the right of \eqref{4.10}, we get $g=\frac{b^2}{4}$. Then \eqref{4.121} means
					\begin{align}
						\label{4.16}
						f''=\frac{b^2}{4}f.
					\end{align}
					Solving the \eqref{4.16}, we have $f(z)=k_1e^{\frac{b}{2}z}+k_2e^{\frac{-b}{2}z}$, where $k_1,k_2$ are  constants. Substituting  $f$ into \eqref{1.6} and using Lemma \ref{lm3.1},  we get $a(\frac{b}{2})^4=-1$,  $8(k_1)^3k_2=p_1$, $8(k_2)^3k_1=p_2$.  Which is  conclusion (i) of Theorem \ref{th1.2}.
				}
			\end{subcase}
		}
	\end{case}

	\vspace{6pt}
	\begin{case}
		\rm{
			$n=3$. From the \eqref{4.10},
			\begin{align}
				\label{4.17}
				f^2(b^2f-3f'')=f'(6ff'-b^2a(f')^2+6a(f'')^2+3af'f''').
			\end{align}
			We can see from \eqref{4.17} the zero of  $f'$ must be the zero of $b^2f-3f''$. Therefore
			\begin{align}
				\label{4.18}
				\frac{b^2f-3f''}{f'}=\frac{6ff'-b^2a(f')^2-6a(f'')^2+3af'f'''}{f^2}=h_1,
			\end{align}
			where $h_1$ is an entire function. By  $T(r,f)=O(r)$  and  Lemma \ref{lm3.4}, implies $m(r,h_1)=o(\log r)$, which means $h_1$ is a constant.
			
			Next we discuss  whether $h_1=0$. Aim for a contradiction.
			
			\vspace{6pt}
			\setcounter{subsection}{2}
			\setcounter{subcase}{0}
			\renewcommand{\thesubcase}{\arabic{subsection}.\arabic{subcase}}
			
			\begin{subcase}
				\rm{
					If $h_1\not=0$. From the \eqref{4.18}, get $f''=\frac{b^2}{3}f-\frac{h_1}{3}f'$, $3f'''=(\frac{h_1^2}{3}+b^2)f'-\frac{b^2h_1}{3}f$. Substituting $f'''$ into \eqref{4.18}, get
					\begin{align}
						\label{4.19}
						f^2(h_1+\frac{2ab^4}{3})=(6+ah_1b^2)ff'-\frac{1}{3}ah_1^2(f'')^2.
					\end{align}
					Let $z_0$ be a zero of $f$, because $f'$ and $f$  don't have common zero,
					then from $f''=\frac{b^2}{3}f-\frac{h_1}{3}f'$, we see $f''(z_0)\not=0$.  Substituting $z_0$ into \eqref{4.19}, then $\frac{1}{3}ah_1^2(f''(z_0))^2=0$.
					Since $a\not=0$ and $h_1\not=0$, which is impossible.
				}
			\end{subcase}

			\vspace{6pt}
			\begin{subcase}
				\rm{
					If $h_1=0$, then from \eqref{4.18}, $3f''=b^2f$. Solving this equation, $f(z)=l_1e^{\frac{b}{\sqrt{3}}z}+l_2e^{\frac{-b}{\sqrt{3}}z}$, where $l_1,l_2$ are nonzero constants.  Substituting  $f$ into \eqref{1.6} and using Lemma \ref{lm3.1}, get $p_1=0$, which is impossible.
				}
			\end{subcase}
			
		}
	\end{case}
	
	\vspace{6pt}
	\setcounter{subsection}{3}
	\setcounter{subcase}{0}
	\renewcommand{\thesubcase}{\arabic{subsection}.\arabic{subcase}}
	\begin{case}
		\rm{
			$n=2$.  If $9ab^2+4\not=0$, then we get Theorem \ref{thE}, which is  the  conclusion $\rm{(ii)-(1)}$ of Theorem \ref{th1.2}. Thus, we only need to consider $9ab^2+4=0$.
			Using \eqref{1.6} and \eqref{4.7} to  eliminate $e^{-bz}$, then
			
			\begin{align}
				\label{4.20}
				bf^2+ab(f')^2+2ff'+2af'f''= 2p_1be^{bz}.
			\end{align}
			Using  \eqref{4.20} and the differential of  \eqref{4.20} to eliminate $e^{bz}$, then
			\begin{align}
				\label{4.21}
				b^2f^2-2ff''+(ab^2-2)(f')^2-2af' f'''-2a(f'')^2=0.
			\end{align}
			From \eqref{4.21}, we have $b^2\frac{f}{f'}-2\frac{f''}{f'}+(ab^2-2)\frac{f'}{f}-2a\frac{f'''}{f}-2a\frac{(f'')^2}{ff'}=0$.
			Then by $T(r,f)=O(r)$  and  Lemma \ref{lm3.4}, we get $m(r,\frac{f}{f'})=o(\log r)$.
			
			Differentiating \eqref{4.21}, then
			\begin{align}
				\label{4.22}
				f'(2b^2f+2(ab^2-3)f''-2af^{(4)})=2f'''(f+3af'').
			\end{align}
			Rewriting \eqref{4.22} as
			\begin{align}
				\label{4.23}
				\frac{2b^2f+2(ab^2-3)f''-2af^{(4)}}{2f'''}=\frac{f+3af''}{f'}=h_2,
			\end{align}
			where $h_2$ is meromorphic  function.   By  $T(r,f)=O(r)$, $m(r,\frac{f}{f'})=o(\log r)$  and  Lemma \ref{lm3.4} implies $m(r,h_2)=o(\log r)$.
			
			We claim $h_2\not\equiv0$. Otherwise, if $h_2\equiv0$, then $f+3af''=0$. Solving this, we get $f(z)=c_1e^{\frac{\sqrt{3}}{2}bz}+c_2e^{-\frac{\sqrt{3}}{2}bz}$, where $c_1,c_2$ are nonzero constants. Substituting  $f$ into \eqref{1.6} and using Lemma \ref{lm3.1}, then $p_1=0$, we get a contradiction.
			
			Differentiating $f'h_2=f+3af''$ twice, get
			\begin{align}
				\label{4.24}
				f'-f''h_2-f'h_2'+3af'''=0,
			\end{align}
			\begin{align}
				\label{4.25}
				f''(1-2h_2')-f'h_2''-f'''h_2+3af^{(4)}=0.
			\end{align}
			
			Combining \eqref{4.23}, \eqref{4.24} and  \eqref{4.25}, we have
			\begin{align}
				\label{4.26}
				\frac{1}{-2}f'H'=f''H,
			\end{align}
			where $\frac{16}{3}+2h_2'+\frac{4h_2^2}{3a}=H$, $H'=2h_2''+\frac{8h_2h_2'}{3a}$.
			
			We claim $H\equiv0$. Otherwise, by the integration of \eqref{4.26}, we have $c(f')^2=\frac{1}{H}$, where $c$ is a nonzero constant.
			By \eqref{4.22}, we see the zero of $f'$ must be the zero of $f'''$ or $f+3af''$. If all the zero of $f'$  are the zero of $f'''$, then from the \cite[proof of Theorem 8 ]{Gao}, we get $ab^2+1=0$ or $ab^2-4=0$, which contradicts $9ab^2=-4$. Therefore, here existence
			a zero $z_0$ of $f'$ such that $f(z_0)+3af''(z_0)=0$, from \eqref{4.23}, we see $z_0$ is not a pole of $h_2$. Since the zero of  $f'$ is simple, then $f''(z_0)\not=0$, substituting $z_0$ into \eqref{4.26}, we get $H(z_0)=0$. Since $c(f')^2=\frac{1}{H}$ and  $f'$ is entire, then $H$ has no zeros, which contradicts $H(z_0)=0$.
			From $H\equiv0$, we get
			\begin{align}
				\label{5.1}
				\frac{16}{3}+2h_2'+\frac{4h_2^2}{3a}=0.
			\end{align}

			If $h_2$ is a non-constant meromorphic function, by sloving  the Riccati equation \eqref{5.1}, we get $h_2(z)=\frac{c_2e^{\frac{8i\sqrt{a}}{3a}z}+2\sqrt{a}}{c_1e^{\frac{8i\sqrt{a}}{3a}z}+i}$ or $h_2(z)=\frac{c_4e^{\frac{-8i\sqrt{a}}{3a}z}+2i\sqrt{a}}{c_2e^{\frac{-8i\sqrt{a}}{3a}z}+1}$.  Where $c_i~(i=1,2,3,4)$ are nonzero constants. By \cite[Theorem 3.12]{HJ}, then $m(r,h_2)=O(\log^2 r)$. which is contradictory to   $m(r,h_2)=o(\log r)$.
			
			If  $h_2$ is a constant. From $9ab^2=-4$ and \eqref{5.1}, we get $h_2=\pm \frac{4}{3b}$. By \eqref{4.23} and $9ab^2+4=0$, we have
			\begin{align}
				\label{4.27}
				f+\frac{-4}{3b^2}f''=f'h_2
			\end{align}
			Solving \eqref{4.27}, then $f(z)=l_1e^{\frac{-3}{2}bz}+l_2e^{\frac{1}{2}bz}$ or $f(z)=l_3e^{\frac{3}{2}bz}+l_4e^{\frac{-1}{2}bz}$, where $l_i(~i=1,2,3,4)$ are nonzero constants. Substituting  $f(z)=l_1e^{\frac{-3}{2}bz}+l_2e^{\frac{1}{2}bz}$ into \eqref{1.6} and using Lemma \ref{lm3.1}, get $l_1^2=\frac{9}{8}p_1$, $l_2^2=\frac{p_2^2}{8p_1}$.
			Substituting  $f(z)=l_3e^{\frac{3}{2}bz}+l_4e^{\frac{-1}{2}bz}$ into \eqref{1.6} and using Lemma \ref{lm3.1}, get $l_3^2=\frac{9}{8}p_2$, $l_4^2=\frac{p_1^2}{8p_2}$. Now we have  the  conclusion $\rm{(ii)-(2)}$ of Theorem \ref{th1.2}.

		}
	\end{case}
	%\end{proof}

	\section*{Declarations}
	\begin{itemize}
		\item \noindent{\bf Funding}
		This research work is supported by the National Natural Science Foundation of China (Grant No. 12261023, 11861023) and Graduate Research Fund Project of Guizhou Province (2024YJSKYJJ186).
		
		\item \noindent{\bf Conflicts of Interest}
		The authors declare that there are no conflicts of interest regarding the publication of this paper.
		
	\end{itemize}

\end{document}